\documentclass[final]{dmtcs-episciences}

\usepackage[utf8]{inputenc}
\usepackage{subfigure}

\usepackage[round]{natbib}

\usepackage{tablefootnote}
\usepackage{caption}
\usepackage{float}
\usepackage{calc}
\usepackage{adjustbox}
\usepackage{relsize}
\usepackage{layout}
\usepackage{geometry}
\usepackage{setspace}
\usepackage{xcolor}
\usepackage{amsmath}
\usepackage{amsthm}
\usepackage{amssymb}
\usepackage{mathtools}
\usepackage{url}

\usepackage{tikz}
\usetikzlibrary{graphs}
\usetikzlibrary{calc,angles}
\usetikzlibrary{arrows,decorations.markings,plotmarks,shapes,positioning}
\usepackage{relsize}

\newtheorem{theorem}{Theorem}
\newtheorem{conjecture}{Conjecture}

\title[On Mixed cages]{On Mixed Cages}
\author[Geoffrey Exoo]{Geoffrey Exoo\affiliationmark{1}}
\affiliation{Indiana State University, Terre Haute, IN, USA}


\begin{document}

\publicationdata
{vol. 25:2}
{2023}
{14}
{10.46298/dmtcs.11057}
{2023-03-10; 2023-03-10; 2023-06-29}
{2023-08-13}

\maketitle

\begin{abstract}
Mixed graphs have both directed and undirected edges.  We consider
mixed graphs that are regular in both directed and undirected degree and
with given girth.
Araujo-Pardo, Hernández-Cruz, and Montellano-Ballesteros recently
established a lower bound for regular mixed graphs with directed degree one that
generalizes the Moore bound.
In this paper we show that, in at least one non-trivial case, their
bound can be achieved.  For a few other specific values of direct degree, undirected
degree and girth we are able to identify graphs (mixed cages) of minimum possible
order.
We also obtain some fairly tight bounds for a few other cases.
\end{abstract}

\section{Introduction}
\label{sec:in}

The cage problem, which asks for the identification of smallest
regular graphs of given girth, has been extensively studied for
graphs and for digraphs.
A series of survey papers have
tracked progress on the problem for graphs \cite{wongsurvey, biggssurvey, cagesurvey}.
The digraph version of the problem was first considered in
1970 by Behzad, Chartrand and Wall \cite{bcw}.
Recently, Araujo-Pardo, Hernández-Cruz, and Montellano-Ballesteros \cite{ahm}
considered the problem for mixed graphs and raised a number of interesting questions.

Finding graphs that give reasonable upper bounds and proving lower bounds both seem
to be difficult problems for the undirected case.
The situation appears to be different for digraphs.
It is generally believed that the extremal digraphs were correctly identified
by Behzad, Chartrand and Wall \cite{bcw}, who proposed a fundamental conjecture
discussed below. For mixed graphs,
as suggested in \cite{ahm}, the level of difficulty in finding
good candidate graphs may depend on the relative sizes
of the directed and undirected degrees.

In this paper, we address and resolve some of the questions
raised in \cite{ahm}.
In particular, it is shown below that the Moore-like bound they established for
directed degree $1$ can be achieved.  In the next section, we review basic
notation and terminology.  Then we review the basic bounds of Moore, of Behzad,
Chartrand and Wall, and of Araujo-Pardo, Hernández-Cruz, and Montellano-Ballesteros.
Finally, we present a number of new results.

\section{Notation and Terminology}
\label{sec:notation}

A mixed graph is a graph with both directed and undirected edges.  We refer to directed
edges as {\em arcs} and undirected edges as {\em edges}.
The sets of vertices, edges, and arcs for a mixed graph $G$ are denoted $V(G)$, $E(G)$, and
$A(G)$, respectively.
The {\em degree} of a vertex $v$ in a
mixed graph $G$ is the number of edges incident with the vertex and is denoted
$deg(v,G)$ or simply $deg(v)$ if $G$ is clear from the context.  Similarly, the {\em in-degree}
and {\em out-degree} of $v$, denoted $ideg(v,G)$ and $odeg(v,G)$, are the numbers of arcs
incident to and from $v$.
A mixed graph $G$ is regular if the degree and out-degree are constant as
$v$ ranges over the vertices of $G$. If, in addition,
the in-degree is constant then the graph is
{\em totally regular}.

A {\em cycle} in a mixed graph is a sequence of vertices $v_0, v_1, \cdots, v_k$ such that
$v_0 = v_k$, and each pair of consecutive vertices $(v_i,v_{i+1})$ is either joined
by an edge or an arc (directed from $v_i$ to $v_{i+1}$), and there are no repeated edges or arcs.
The {\em girth} of a mixed graph is the length of a shortest cycle.  Note that the definition
considers the possibility of $1$-cycles (loops) and $2$-cycles.  As usual, we use $C_k$ to
denote an undirected $k$-cycle, and $\vec{C}_k$ to denote a directed $k$-cycle.

An $(r,z,g)$-graph\footnote{This notation is compatible with most of the degree/diameter literature.
In \cite{ahm} a notation is used that reverses the order of $r$ and $z$.}
is a regular mixed graph with degree $r$, out-degree $z$, and girth $g$.
We are interested in cases where $r,z > 0$ and $g \geq 5$.
An $(r,z,g)$-cage is an $(r,z,g)$-graph of minimum possible order.
We denote this minimum order by $f(r,z,g)$.

For graphs or digraphs $G$ and $H$, the
{\it graph composition} $G[H]$ is the graph with vertex set $V(G) \times V(H)$ where
$(g_1,h_1)$ is adjacent to $(g_2,h_2)$ if $g_1$ is adjacent to $g_2$ or if $g_1 = g_2$ and
$h_1$ is adjacent to $h_2$ \cite{harary}.  This is also known as the {\it lexicographic product}.

Finally, we note that modular arithmetic on subscripts will frequently be employed
when specifying constructions.
In obvious cases, we will try to avoid cluttering the presentation
with frequent reminders.
In one case we use expressions of the form $a + b\;mod\; c$ and
remind the reader that the modulus operator has the precedence
of multiplication.

\section{Basic Bounds}
\label{sec:basic}

The following conjecture is fundamental to the cage problem for directed
graphs.

\begin{conjecture}[Behzad-Chartrand-Wall, \cite{bcw}]
The order of a smallest $r$ regular digraph of girth $g$ is $n = r(g-1)+1$.
\end{conjecture}

Behzad, Chartrand and Wall also identified the digraphs, for each $r$ and $g$, that they believed to be extremal,
These digraphs, which we denoted $BCW(r,g)$, have orders $r(g-1)+1$ with vertices $v_i$
for $0 \leq i \leq r(g-1)$
and arcs $v_iv_{i+j}$ for $1 \leq j \leq r$.
A small example is shown in Figure~\ref{figbcw}.

\begin{figure}[htbp]
\centering
\setlength{\abovecaptionskip}{25pt}
\tikzset{fontscale/.style = {font=\relsize{#1}}}
\begin{tikzpicture}[scale=1.0,line width=1pt]
\tikzstyle{arc}=[thick,
                 decoration={markings,mark=at position 0.70 with {\arrow[scale=1.0,>=triangle 45]{>}}},
                 postaction={decorate}];

\tikzstyle{fred}=[draw=black, fill=yellow, thick,
  shape=circle, minimum height=2.5mm, inner sep=1, text=black, font=\bfseries];

\newcommand\orad{3}
\newcommand\rotd{27.6923077}

\begin{scope}[rotate=90]

\foreach \rot in {0,1,2,...,12}
{
    \begin{scope}[rotate=\rot*\rotd]
        \draw[arc] (\rotd:\orad) -- (0:\orad);
        \draw[arc] (\rotd*2:\orad)-- (0:\orad);
        \draw[arc] (\rotd*3:\orad)-- (0:\orad);
    \end{scope}
}

\foreach \rot in {0,1,2,...,12}
{
    \begin{scope}[rotate=\rot*\rotd]
        \node[fred,fill=yellow!50!white] at (0:\orad) [fontscale=2] {};
    \end{scope}
}

\end{scope}

\end{tikzpicture}
\caption{The Behzad-Chartrand-Wall Graph for degree $3$ and girth $5$ has order $13$.}
\label{figbcw}
\end{figure}
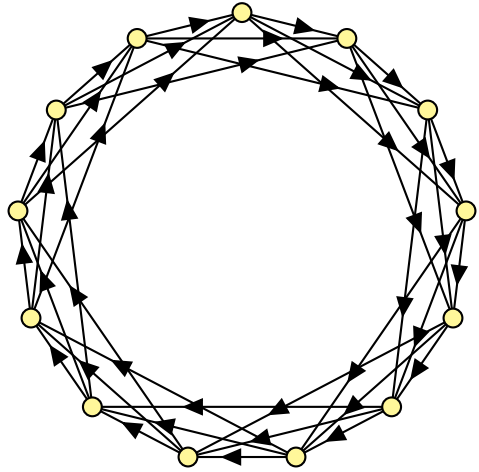

In \cite{ahm} 
Araujo-Pardo, Hernández-Cruz, and Montellano-Ballesteros considered
the problem of finding mixed cages.
They focused on the case $z=1$ and found a lower bound for $f(r,1,g)$
based on the well known Moore bound for undirected graphs.
Their idea is to
attach undirected Moore trees to each vertex of a directed path of order $g$,
choosing trees whose depth is as large as possible while still guaranteeing that
all tree vertices are distinct.


Recall that the Moore bound for an $r$-regular graph of diameter $d$ is given by:

\begin{equation}
n(r,d) \ge \frac{r(r-1)^{d}-2}{r-2}
\end{equation}

Let $v_0, v_1, \cdots , v_{g-1}$ be the vertices of a path of order $g$.
Attach a Moore tree of depth $i$ to $v_i$ and $v_{g-1-i}$ (if distinct),
for $0 \leq i \leq \lfloor g/2 \rfloor $.
This gives the following bound.

\begin{theorem}[The AHM Bound]

\begin{equation*}
f(r,1,g) \geq \mathlarger{\mathlarger{\sum}}_{i=0}^{g-1}n(r,\, {\tt min}(i,g{-}i{-}1))
\end{equation*}

\end{theorem}

\medskip

A small example AHM tree is shown in Figure~\ref{figahm}.  We will see that this bound
can, at least occasionally, be achieved.  A central question for mixed cages is to
determine how often it is achieved.

\begin{figure}[ht]
\centering
\begin{tikzpicture}[scale=0.5]
\tikzstyle{vertex}=[draw=black, fill=yellow!50!white, thick, shape=circle, inner sep=0, minimum height=4.0]

    \foreach \xco in {0,4,...,16}
    {
        \draw[thick,decoration={markings,mark=at position 0.7 with
            {\arrow[scale=0.8,>=triangle 45]{>}}},postaction={decorate}] (\xco,6) -- (\xco+4,6);
    }
    \foreach \xco in {4,8,12,16} {
        \draw[line width=1.0,color=black] (\xco,6) -- (\xco,4);
        \draw[line width=1.0,color=black] (\xco,6) -- (\xco-1.2,4);
        \draw[line width=1.0,color=black] (\xco,6) -- (\xco+1.2,4);
    }
    \foreach \xco in {8,12}
    {
        \draw[line width=1.0,color=black] (\xco-1.2,4) -- (\xco-1.5,2);
        \draw[line width=1.0,color=black] (\xco-1.2,4) -- (\xco-0.9,2);
        \draw[line width=1.0,color=black] (\xco,4) -- (\xco-0.3,2);
        \draw[line width=1.0,color=black] (\xco,4) -- (\xco+0.3,2);
        \draw[line width=1.0,color=black] (\xco+1.2,4) -- (\xco+0.9,2);
        \draw[line width=1.0,color=black] (\xco+1.2,4) -- (\xco+1.5,2);
    }
    \foreach \xco in {0,4,...,20}
    {
        \node[vertex] at (\xco,6) {};
    }
    \foreach \xco in {4,8,12,16}
    {
        \node[vertex] at (\xco,4) {};
        \node[vertex] at (\xco-1.2,4) {};
        \node[vertex] at (\xco+1.2,4) {};
    }
    \foreach \xco in {8,12}
    {
        \node[vertex] at (\xco-1.5,2) {};
        \node[vertex] at (\xco-0.9,2) {};
        \node[vertex] at (\xco-0.3,2) {};
        \node[vertex] at (\xco+0.3,2) {};
        \node[vertex] at (\xco+0.9,2) {};
        \node[vertex] at (\xco+1.5,2) {};
    }
\end{tikzpicture}
\caption{The AHM Tree for $r=3$, $z=1$, $g=6$.}
\label{figahm}
\end{figure}
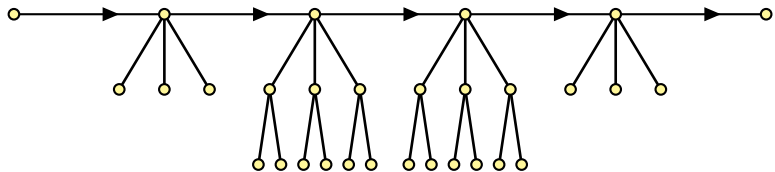

\medskip

Upper bounds on $f(r,1,g)$ can be obtained from
known cages (or candidates) of degree $r+2$ and girth $g$ whenever the
graph is Hamiltonian (or at least has a $2$-factor with no short cycles).
We shall refer to the upper bound obtained this
way as the {\em Cage Bound}.  In general, the Cage Bound is probably
weak, but in at least one case it may be useful.
The Hoffman-Singleton graph is
the $(7,5)$-cage, and is Hamiltonian.  If any Hamiltonian cycle is
cyclically oriented, a $(5,1,5)$ mixed graph of order $50$ is obtained.

It appears that most cages are Hamiltonian.  A theorem of Singer \cite{singer}
on collineations implies that girth $6$ cages constructed from field planes
are Hamiltonian.   More recently, Lazebnik, Mellinger and Vega \cite{lmv}
showed that the incidence graphs of all finite projective planes are Hamiltonian.
This author has tested all the remaining
cages and cage candidates
listed in Table 6 of the Dynamic Cage Survey \cite{cagesurvey} that are
relevant to the results in this paper and found them all to be Hamiltonian.
The cage bound may at least serve as a useful benchmark.  One might view finding
$(r,1,g)$-graphs whose orders are less that the orders of the
$(r{+}2,g)$-cages (or candidates) as significant first steps.

\section{The Case $r + z \leq 3$}
\label{sec:caserz3}

In this section, mixed cages for $r,z \leq 2$ are determined.
The first theorem uses what could be viewed as a degenerate case of
the AHM bound, as it employs Moore trees of degree one, which are
single edges.

\begin{theorem}
\begin{equation*}
f(1,1,g) =  2g-2
\end{equation*}
\end{theorem}
\begin{proof}
The lower bound follows from the AHM bound.  The AHM bound is achieved by
the M\"{o}bius ladder, where the outer cycle (see Figure~\ref{figmob}) is
cyclically oriented.
\end{proof}

\begin{figure}[ht]
\centering
\begin{tikzpicture}[scale=0.7]
\tikzstyle{vertex}=[draw=black, fill=yellow!50!white, thick, shape=circle, inner sep=0, minimum height=4.0]
    \foreach \rot in {0,45,...,315}
    {
        \draw[thick,decoration={markings,mark=at position 0.7 with
            {\arrow[scale=0.8,>=triangle 45]{>}}},postaction={decorate}] (\rot:2) -- (\rot+45:2);
    }
    \foreach \rot in {0,45,90,135}
    {
        \draw[line width=1.0,color=black] (\rot:2) -- (\rot+180:2);
    }
    \foreach \rot in {0,45,...,315}
    {
        \node[vertex] at (\rot:2) {};
    }
\end{tikzpicture}
\caption{The M\"{o}bius ladder of order $8$: the $(1,1,5)$-cage.  An AHM tree can be
found by using any directed path of length four and the three edges incident
with the non-endvertices of the path.}
\label{figmob}
\end{figure}
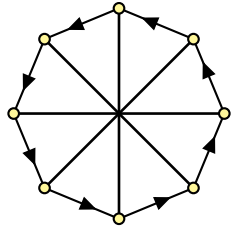

In \cite{ahm} the authors consider and resolve the case, $r=2, z=1$.

\begin{theorem}[AHM \cite{ahm}]
\begin{equation*}
f(2,1,g) =  \begin{cases}
\dfrac{g^2+1}{2} &, \text{  $g$ odd } \\
\dfrac{g^2}{2} &, \text{  $g$ even }
\end{cases}
\end{equation*}
\end{theorem}

The extremal graphs in \cite{ahm} are derived from the Behzad, Chartrand and Wall digraphs by
replacing one directed spanning cycle by an undirected cycle.
The lower bound follows from the AHM bound, giving the first examples where the bound
is achieved.
 
The value of $f(1,2,g)$ can be determined using the $r=2$ case of
Behzad-Chartrand-Wall Conjecture, proved by Behzad, and by the graph $\vec{C}_g[K_2]$.

\begin{theorem}
\begin{equation*}
f(1,2,g) =  2g
\end{equation*}
\end{theorem}
\begin{proof}
Let $G$ be a $(1,2,g)$-graph.
Consider the directed subgraph.  The result of Behzad \cite{behzad} on the case $r=2$ of the
Behzad-Chartrand-Wall Conjecture implies that $G$
must have at least $2g-1$ vertices.  But since $r$ is odd, the order of $G$ must be even.
So $f(1,2,g) \geq  2g$.
The extremal graph is $\vec{C}_g[K_2]$.  The graph $\vec{C}_7[K_2]$ is shown in
Figure~\ref{g127fig}.
\end{proof}

\medskip

\begin{figure}[htbp]
\centering
\setlength{\abovecaptionskip}{25pt}
\tikzset{fontscale/.style = {font=\relsize{#1}}}
\begin{tikzpicture}[scale=1.0,line width=1pt]

\tikzstyle{fred}=[draw=black, fill=yellow, thick,
  shape=circle, minimum height=3mm, inner sep=1, text=black, font=\bfseries];
\tikzstyle{arc}=[thick,
                 decoration={markings,mark=at position 0.85 with {\arrow[scale=1.0,>=triangle 45]{>}}},
                 postaction={decorate}];

\newcommand\orad{3.5}
\newcommand\mrad{2.0}
\newcommand\rotb{51.4286}

\begin{scope}[rotate=90]

\foreach \rot in {0,1,2,...,6}
{
    \begin{scope}[rotate=\rot*\rotb]
        \draw[line width=1.5,brown!60!black] (0:\mrad) -- (0:\orad);
        \draw[arc] (\rotb:\orad) -- (0:\orad);
        \draw[arc] (\rotb:\orad) -- (0:\mrad);
        \draw[arc] (\rotb:\mrad) -- (0:\orad);
        \draw[arc] (\rotb:\mrad) -- (0:\mrad);
    \end{scope}
}

\foreach \rot in {0,1,2,...,6}
{
    \begin{scope}[rotate=\rot*\rotb]
        \node[fred,fill=yellow!50!white] at (0:\orad) [fontscale=2] {};
        \node[fred,fill=yellow!50!white] at (0:\mrad) [fontscale=2] {};
    \end{scope}
}

\end{scope}

\end{tikzpicture}
\caption{The graph $\vec{C}_7[K_2]$ of order $14$ with $(r,z,g) = (1,2,7)$.}
\label{g127fig}
\end{figure}
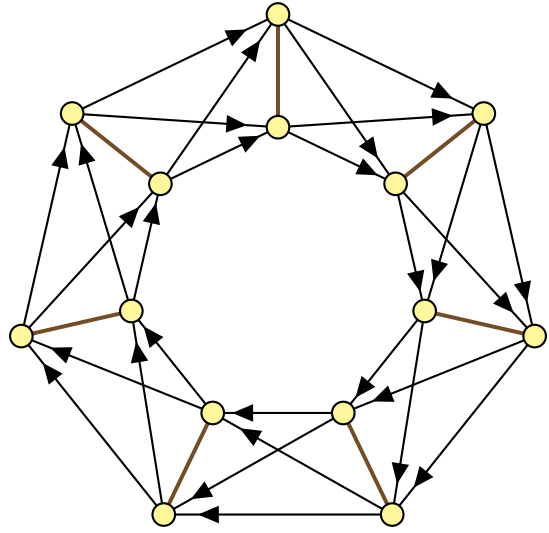

\section{The Case $r=2$}
\label{sec:caser2}

In this section we present three theorems that provide upper bounds for $f(2,z,g)$.
When $r=2$, a simple upper bound can be obtained using subgraphs of $\vec{C}_g[C_g]$.
\cite{bcw,ahm}.

\begin{theorem}
\label{thm2zga}
\begin{equation*}
f(2,z,g) \leq g^2, \mbox{ for } z \leq g.
\end{equation*}
\end{theorem}
\begin{proof}
The required graphs are all subgraphs of $\vec{C}_g[C_g]$.
Construct a $(2,z,g)$-graph $G$ as follows.
The vertices of $G$ are $\{ v_{i,j} \,|\, 0 \leq i,j < g \}$,
with edges $v_{i,j}v_{i,j+1}$, 
and arcs $v_{i,j}v_{i+1,j+k}$ for $0 \leq k < z$ (subscript arithmetic is modulo $g$).
\end{proof}

For a restricted range of $z$, we can remove one of the $g$-cycles and still maintain the girth.

\begin{theorem}
\label{thm2zgb}
\begin{equation*}
f(2,z,g) =  \begin{cases}
g^2-g, & \text{ $g$ even and $z \leq g/2$} \\[10pt]
g^2-1, & \text{ $g$ odd and $z \leq (g{+}1)/2$}
\end{cases}
\end{equation*}
\end{theorem}
\begin{proof}
If $g$ is even, let $g=2h$, and then $z \leq h$.
Construct a $(2,z,g)$-graph $G$ of order $g^2 - g$ as follows.
Let $V(G) = \{ v_{i,j} \,|\, 0 \leq i < g-1 \text{ and } 0 \leq j < g \}$,
$E(G) = \{v_{i,j}v_{i,j+1} | \, 0 \leq j < g-1 \text{ and } 0 \leq j < g \}$.
The arc set is a bit more complex:
\begin{equation*}
\begin{aligned}
A(G) = & \{ v_{i,j} v_{i+1,(j+k)\,mod\,h} |\, 0 \leq i < g{-}2, \; 0 \leq j < h, \, 0 \leq k < z \} \; \cup \\
       & \{ v_{i,h+j} v_{i+1,h+(j+k)\,mod\,h} |\, 0 \leq i < g{-}2, \; 0 \leq j < h, \, 0 \leq k < z \} \; \cup \\
       & \{ v_{g-2,j} v_{0,h+(j+k)\,mod\,h} |\, 0 \leq j < h, \; 0 \leq k < z \} \; \cup \\
       & \{ v_{g-2,j+h} v_{0,(j+k)\,mod\,h} |\, 0 \leq j < h, \; 0 \leq k < z \}.
\end{aligned}
\end{equation*}

\end{proof}

A small example of the construction in Theorem~\ref{thm2zgb} is shown in Figure~\ref{mobfig}.

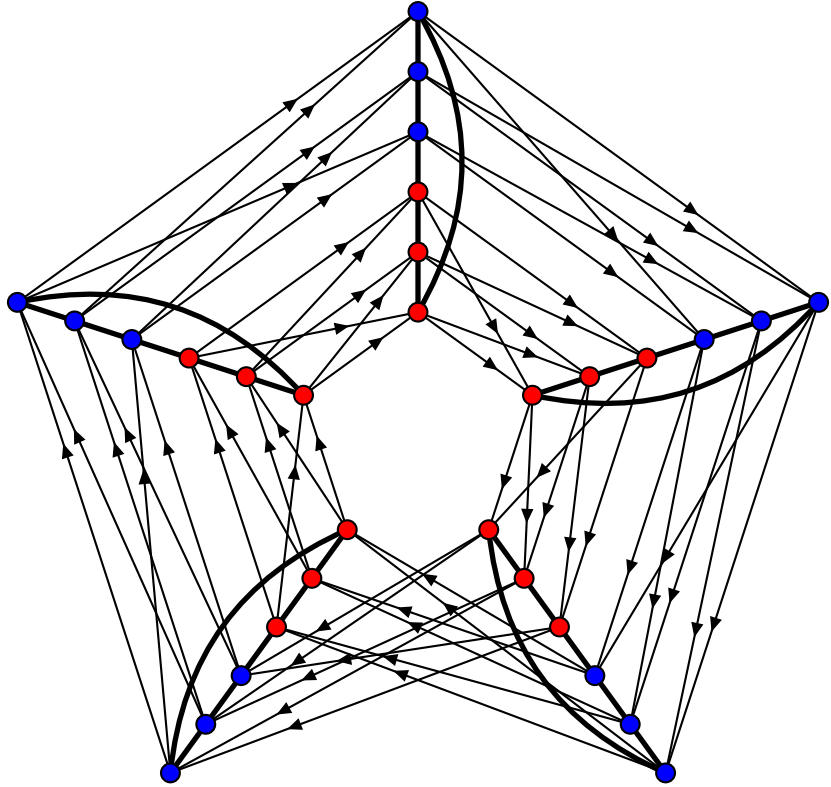
\begin{figure}[htbp]
\centering
\setlength{\abovecaptionskip}{25pt}
\tikzset{fontscale/.style = {font=\relsize{#1}}}
\begin{tikzpicture}[scale=0.8,line width=1pt]

\tikzstyle{fred}=[draw=black, fill=yellow, thick,
  shape=circle, minimum height=0.25cm, inner sep=1, text=black, font=\bfseries];
\tikzstyle{arc}=[thick,
                 decoration={markings,mark=at position 0.70 with {\arrow[scale=0.7,>=triangle 45]{>}}},
                 postaction={decorate}];

\begin{scope}[rotate=90]


    \foreach \rot in {0,1,3,4}
    {
        \begin{scope}[rotate=\rot*72]
            \draw[arc] (72:2) -- (0:2);
            \draw[arc] (72:2) -- (0:3);
            \draw[arc] (72:3) -- (0:3);
            \draw[arc] (72:3) -- (0:4);
            \draw[arc] (72:4) -- (0:4);
            \draw[arc] (72:4) -- (0:2);
            \draw[arc] (72:5) -- (0:5);
            \draw[arc] (72:5) -- (0:6);
            \draw[arc] (72:6) -- (0:6);
            \draw[arc] (72:6) -- (0:7);
            \draw[arc] (72:7) -- (0:7);
            \draw[arc] (72:7) -- (0:5);
        \end{scope}
    }
    \begin{scope}[rotate=144]
            \draw[arc] (72:2) -- (0:5);
            \draw[arc] (72:2) -- (0:6);
            \draw[arc] (72:3) -- (0:6);
            \draw[arc] (72:3) -- (0:7);
            \draw[arc] (72:4) -- (0:7);
            \draw[arc] (72:4) -- (0:5);
            \draw[arc] (72:5) -- (0:2);
            \draw[arc] (72:5) -- (0:3);
            \draw[arc] (72:6) -- (0:3);
            \draw[arc] (72:6) -- (0:4);
            \draw[arc] (72:7) -- (0:4);
            \draw[arc] (72:7) -- (0:2);
    \end{scope}


    \foreach \rot in {0,1,2,3,4}
    {
        \begin{scope}[rotate=\rot*72]
            \draw[line width=2.0] (0:2) -- (0:7);
            \draw[line width=2.0] (0:2) to [out=330,in=210] (0:7);
        \end{scope}
    }


    \foreach \rot in {0,1,2,...,4}
    {
        \begin{scope}[rotate=\rot*72]
            \foreach \rad/\c in {2/red,3/red,4/red,5/blue,6/blue,7/blue}
            {
                \node[fred,fill=\c] at (0:\rad) {};
            }
        \end{scope}
    }

\end{scope}
\end{tikzpicture}
\caption{A $(2,2,6)$-graph with $5$ undirected $6$-cycles.}
\label{mobfig}
\end{figure}

Finally, when $r = z = 2$ we reduce the coefficient of $g^2$ in the upper bound to $3/4$.

\begin{theorem}
\begin{equation*}
f(2,2,g) \leq \bigg \lceil \frac{g}{2} \bigg \rceil \bigg \lfloor \frac{3g}{2} \bigg \rfloor = \left\{
\begin{array}{ll} \frac{3g^2}{4} & g \mbox{ even, } \\[10pt]
                          \frac{3g^2+2g-1}{4} & g \mbox{ odd } 
\end{array}
\right.
\end{equation*}
\end{theorem}
\begin{proof}
Let $s = \Big \lceil \frac{g}{2} \Big \rceil$, and $t = \Big \lfloor \frac{3g}{2} \Big \rfloor$.
We describe a graph $G$ of order $st$ with $r=z=2$, and girth $g$.
Let
$$
V(G) = \{ v_{i,j} \,|\, 0 \leq i < s,  0 \leq j < t \}.
$$
The undirected subgraph of $G$ consists of $s$ disjoint $t$-cycles,
where the first subscript of a vertex indicates
which cycle the vertex is on, and the second subscript indicates its position on the cycle.
We imagine the cycles are numbered according to the first coordinate their vertices, and refer to
cycle $0$, cycle $1$, on up to cycle $s-1$, with the vertices of each cycle numbered in the
natural order.

Next we specify the arcs so that all arcs from vertices on cycle $i$ go to vertices on cycle $i+1$.
So
$v_{i_1,j_1} v_{i_2,j_2} \in A(G)$ if $i_1 + 1 = i_2 < s$ and either $j_1 = j_2$ or $j_1 + 1 = j_2$.

Finally the arcs from cycle $s{-}1$ back to cycle $0$:
$v_{s-1,j_1} v_{0,j_2}$ is an arc if
$j_1 + \Big \lfloor g/2 \Big \rfloor = j_2$ or
$j_1 + \Big \lfloor g/2 \Big \rfloor + 1 = j_2$.

Note that the function that maps $v_{i,j}$ to $v_{i,j+1}$ is a graph automorphism.

Any possible cycle of length less than $g$ must contain arcs, since the undirected subgraph consists
of disjoint $t$-cycles.  But any cycle containing an arc contains at least $t$ arcs and at least one
arc joining each pair of consecutive $t$-cycles.  Hence it also contains at least one
vertex from each of the undirected cycles.
So by symmetry, if there is a cycle of length less than $g$ there is such a cycle containing vertex $0$.
But any such cycle must contain
exactly $r$ arcs and fewer than $g-r = \lceil \frac{g}{2} \rceil$ edges.
By construction, this is impossible.
\end{proof}

Note that if this construction were modified for $f(2,3,g)$ we would need undirected
cycles of length $2g$ and would
no longer have an improvement over the upper bound in Theorem~\ref{thm2zga}.

For the case $(2,2,5)$, a graph of order $19$ was found by
Claudia De La Cruz and Miguel Piza\~{n}a (personal communication).
Their graph can viewed as a mixed orientation of the cubic residue graph.
The graph in Figure~\ref{fig22g} is a smallest $(2,2,6)$-graph.

\begin{figure}[htbp]
\centering
\setlength{\abovecaptionskip}{10pt}
\begin{tikzpicture}[scale=0.75,line width=1pt]
\tikzstyle{fred}=[draw=black, fill=yellow, very thick,
  shape=circle, minimum height=0.3cm, inner sep=0, text=black]
\tikzstyle{arc}=[very thick,
                 decoration={markings,mark=at position 0.70 with {\arrow[scale=0.8,>=triangle 45]{>}}},
                 postaction={decorate}];
\tikzstyle{redarc}=[very thick,red,
                 decoration={markings,mark=at position 0.50 with {\arrow[scale=1.0,>=triangle 45]{>}}},
                 postaction={decorate}];

\newcommand\orad{5}
\newcommand\irad{3}
\newcommand\rota{12.857143}
\newcommand\rotb{25.7143}
\newcommand\rotc{38.57142}

\begin{scope}[rotate=90]
\foreach \rot in {0,1,2,...,9}
{
    \begin{scope}[rotate=\rot*40]
        \draw[line width=1.0,black] (0:2) -- (40:2);
        \draw[line width=1.0,black] (0:4) -- (40:4);
        \draw[line width=1.0,black] (-20:3) -- (20:3);
        \draw[arc] (0:4) -- (20:3);
        \draw[arc] (0:4) -- (-20:3);
        \draw[arc] (20:3) -- (0:2);
        \draw[arc] (-20:3) -- (0:2);
    \end{scope}
}

\draw[redarc] (0:2) to [out=135,in=-15] (160:4);
\draw[redarc] (0:2) to [out=225,in=15] (200:4);

\foreach \rot in {0,1,2,...,8}
{
    \begin{scope}[rotate=\rot*40]
        \node[fred,fill=green] at (0:2) {};
        \node[fred,fill=red] at (20:3) {};
        \node[fred,fill=blue] at (0:4) {};
    \end{scope}
}
\end{scope}
\end{tikzpicture}
\caption{A $(2,2,6)$-graph of order $27$.  Only one pair
of arcs from the green cycle to the blue cycle are shown. The others are obtained by rotation.}
\label{fig22g}
\end{figure}
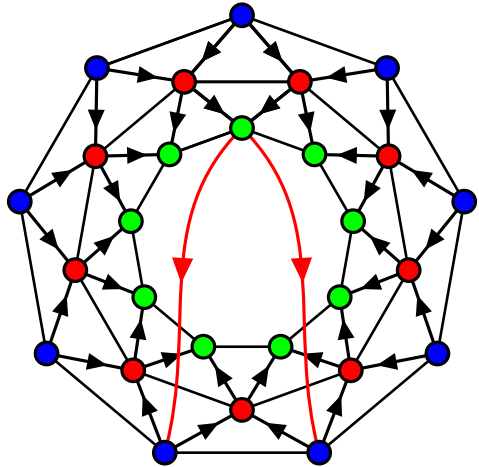

\section{The Case $z=1$}
\label{sec:casez1}

As noted in \cite{ahm}, the cases where $z=1$ may result in graphs most similar to undirected
cages.  In these cases the directed subgraph consists of a set of disjoint cycles whose lengths
partition the order of the graph.

The case $(3,1,5)$ was considered in
\cite{ahm} where they established the bounds $20 \leq f(3,1,5) \leq 28$.
We settle this case below.

\begin{theorem}
f(3,1,5) = 24
\end{theorem}
\begin{proof}
The AHM bound in this case is $20$, so one needs to show that
$(3,1,5)$-graphs of orders $20$ and $22$ do not exist.  This was done by
exhaustive computer searches, which made heavy use of the {\it nauty and Traces}
\cite{nauty} package.
The searches were completed by two different programs.

The first program breaks the task into cases based on the structure of
the (maximal) directed subgraph of a putative $(3,1,5)$-graph.
This subgraph must be a union of cycles.  The lengths of these
cycles give a partition of $20$ (respectively $22$) with no part less than $5$.
There are $13$ (resp. $18$) such partitions.  For each case, the program then
does an exhaustive search for an undirected subgraph.

The other program reverses the procedure, and divides the search into
cases according to the structure of the undirected subgraph.
This subgraph must be a cubic graph whose girth is at least $5$.
There are $5784$ such graphs of order $20$ and
$90940$ of order $22$ \cite{nauty}.
While this leaves more cases to consider than the
first program, the exhaustive search for each case is much faster.

After completing the searches for $n=20$ and $n=22$, and finding no graphs,
the first of these methods was applied to the case $n=24$, and $23$ different graphs
were generated.
The most symmetric of these graphs is depicted in Figure~\ref{g315fig}.
\end{proof}


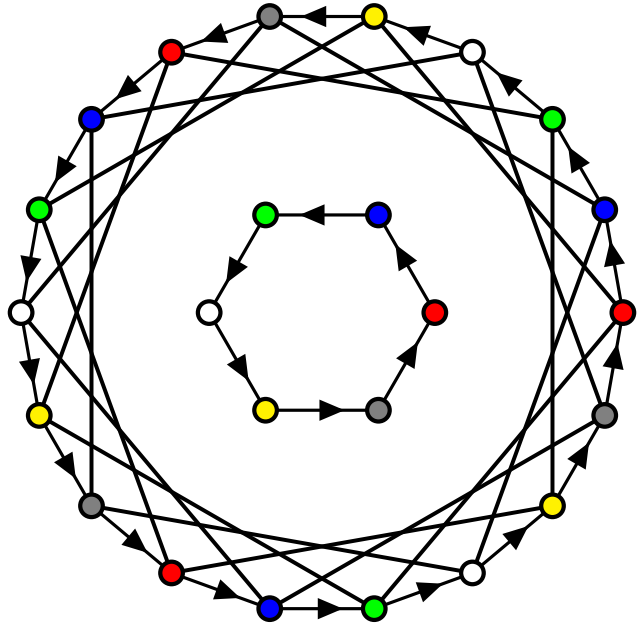
\begin{figure}[H]
\centering
\setlength{\abovecaptionskip}{20pt}
\tikzset{fontscale/.style = {font=\relsize{#1}}}
\begin{tikzpicture}[scale=1.0,line width=1pt]

\tikzstyle{fred}=[draw=black, fill=yellow, ultra thick,
  shape=circle, minimum height=0.3cm, inner sep=1, text=black, font=\bfseries];
\tikzstyle{arc}=[very thick,
                 decoration={markings,mark=at position 0.70 with {\arrow[scale=1.0,>=triangle 45]{>}}},
                 postaction={decorate}];

\newcommand\orad{4}
\newcommand\irad{1.5}

\foreach \theta in {0,20,40,...,340}
{
  \begin{scope}[rotate=\theta]
    \draw[arc] (0:\orad) -- (20:\orad);
    \draw[line width=1.50, black] (0:\orad) -- (80:\orad);
  \end{scope}
}
\foreach \theta in {0,60,120,...,300}
{
  \begin{scope}[rotate=\theta]
    \draw[arc] (0:\irad) -- (60:\irad);
  \end{scope}
}

\foreach \rot in {0,120,240}
{
    \begin{scope}[rotate=\rot]
        \foreach \v/\c in {0/red,1/blue,2/green,3/white,4/yellow,5/gray}
        {
        \begin{scope}[rotate=\v*20]
            \node[fred,fill=\c] at (0:\orad) [fontscale=2] {};
        \end{scope}
        }
    \end{scope}
}

\foreach \v/\c in {0/red,1/blue,2/green,3/white,4/yellow,5/gray}
{
\begin{scope}[rotate=\v*60]
    \node[fred,fill=\c] at (0:\irad) [fontscale=2] {};
\end{scope}
}

\end{tikzpicture}
\caption{A graph of order $24$ with $(r,z,g) = (3,1,5)$.  There are
edges joining inner vertices to all outer vertices of the same color.}
\label{g315fig}
\end{figure}


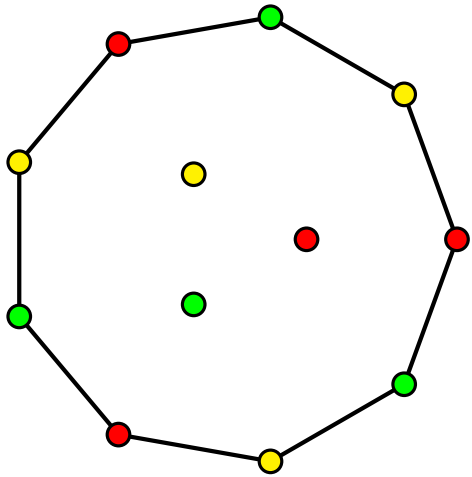
\begin{figure}[H]
\centering
\setlength{\abovecaptionskip}{12pt}
\tikzset{fontscale/.style = {font=\relsize{#1}}}
\begin{tikzpicture}[scale=1.0,line width=1pt]

\tikzstyle{fred}=[draw=black, fill=yellow, very thick,
  shape=circle, minimum height=0.3cm, inner sep=1, text=black, font=\bfseries];
\tikzstyle{arc}=[very thick,
                 decoration={markings,mark=at position 0.70 with {\arrow[scale=1.0,>=triangle 45]{>}}},
                 postaction={decorate}];

\newcommand\orad{3}
\newcommand\irad{1}

\foreach \theta in {0,40,...,320}
{
  \begin{scope}[rotate=\theta]
    \draw[ultra thick, black] (0:\orad) -- (40:\orad);
  \end{scope}
}

\foreach \rot in {0,120,240}
{
    \begin{scope}[rotate=\rot]
        \foreach \v/\c in {0/red,1/yellow,2/green}
        {
        \begin{scope}[rotate=\v*40]
            \node[fred,fill=\c] at (0:\orad) [fontscale=2] {};
        \end{scope}
        }
    \end{scope}
}

\foreach \v/\c in {0/red,1/yellow,2/green}
{
\begin{scope}[rotate=\v*120]
    \node[fred,fill=\c] at (0:\irad) [fontscale=2] {};
\end{scope}
}

\end{tikzpicture}
\caption{One of the two cubic graphs of order $12$ with girth $5$.
Each of the three inner vertices is adjacent to the three
vertices on the $9$-cycle having the matching color.}
The undirected subgraph in Figure~\ref{g315fig} is isomorphic to two
copies of this graph.
\label{g315sub}
\end{figure}

The undirected subgraph of the graph in Figure~\ref{g315fig} is isomorphic to two
copies of the cubic graph of girth $5$ and order $12$ shown in Figure~\ref{g315sub}.
An alternate view of the $(3,1,5)$-graph is given in Figure~\ref{alt315}.
In this graph we fix an ordering of the colors {\em red, blue, yellow, white, green, gray},
and refer to
{\em red} as color 0, {\em blue} as color 1, etc.  The directed edges of the graph are given
as follows.  For each of the three central vertices (in both subgraphs), we add an arc from
the vertex in color $i$ to the vertex (in the other subgraph) in color $i+1 \pmod{6}$.
This gives the directed $6$-cycle from Figure~\ref{g315fig}.

Next pick any red (color $0$) vertex on the outer $9$-cycle of the left subgraph, call it $v_0$.
Then pick any blue (color $1$) vertex on the outer $9$-cycle of the right subgraph, call it $v_1$.
Add an arc from $v_0$ to $v_1$.  Next let $v_2$ be the yellow (color $2$) vertex in the left
subgraph adjacent to $v_0$.  Add an arc from $v_1$ to $v_2$.
Then $v_3$ is the white (color $3$) vertex in the right subgraph adjacent to $v_1$, and add an
arc from $v_2$ to $v_3$.
Continue in this manner to
complete the directed $18$ cycle as in Figure~\ref{g315fig}.

\begin{figure}[H]
\centering
\setlength{\abovecaptionskip}{12pt}
\begin{tikzpicture}[scale=1.0,line width=1pt]
\tikzstyle{fred}=[draw=black, fill=yellow, ultra thick,
  shape=circle, minimum height=4mm, inner sep=0, text=black];
\tikzstyle{edge}=[draw=black, thick];
\tikzstyle{arc}=[ultra thick,
                 decoration={markings,mark=at position 0.70 with {\arrow[scale=0.7,>=triangle 45]{>}}},
                 postaction={decorate}];

\newcommand\orad{3.0}
\newcommand\irad{1.0}

\foreach \theta in {0,40,...,320}
{
  \begin{scope}[rotate=\theta]
    \draw[fred] (0:\orad) -- (40:\orad);
  \end{scope}
}

\foreach \rot in {0,120,240}
{
    \begin{scope}[rotate=\rot]
        \foreach \v/\c in {0/red,1/yellow,2/green}
        {
        \begin{scope}[rotate=\v*40]
            \node[fred,fill=\c] at (0:\orad) {};
        \end{scope}
        }
    \end{scope}
}

\foreach \v/\c in {0/red,1/yellow,2/green}
{
\begin{scope}[rotate=\v*120]
    \node[fred,fill=\c] at (0:\irad) {};
\end{scope}
}
\begin{scope}[xshift=8cm]
\foreach \theta in {0,40,...,320}
{
  \begin{scope}[rotate=\theta]
    \draw[fred] (0:\orad) -- (40:\orad);
  \end{scope}
}

\foreach \rot in {0,120,240}
{
    \begin{scope}[rotate=\rot]
        \foreach \v/\c in {0/gray,1/blue,2/white}
        {
        \begin{scope}[rotate=\v*40]
            \node[fred,fill=\c] at (0:\orad) {};
        \end{scope}
        }
    \end{scope}
}

\foreach \v/\c in {0/gray,1/blue,2/white}
{
\begin{scope}[rotate=\v*120]
    \node[fred,fill=\c] at (0:\irad) {};
\end{scope}
}
\end{scope}
\end{tikzpicture}
\caption{An alternate view of the $(3,1,5)$-graph.}
\label{alt315}
\end{figure}
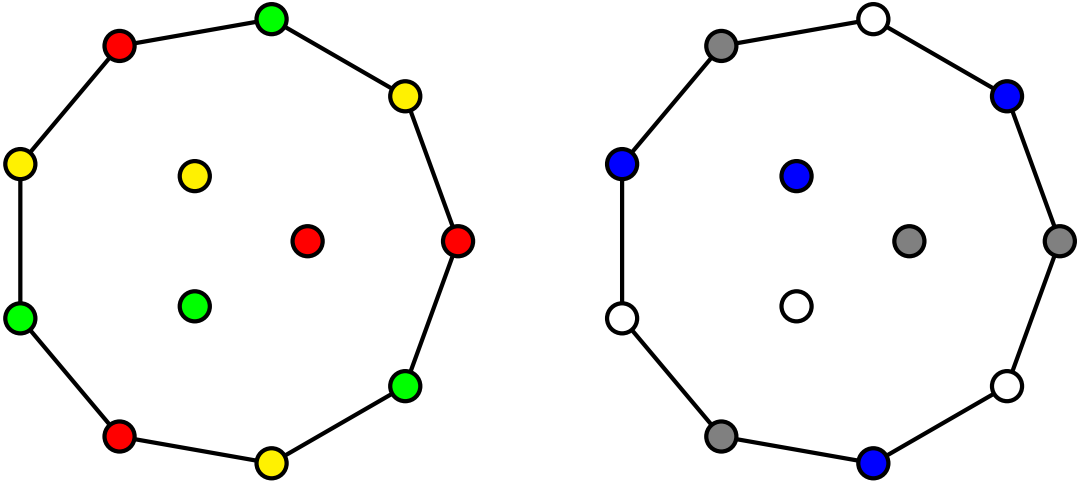

The next construction was presented in \cite{ref316} and shows that the AHM bound can be achieved
in at least one case where $r > 2$.

\begin{theorem}
\label{thm316}
$f(3,1,6) = 30$
\end{theorem}
\begin{proof}
In this case the AHM bound is $30$.
To construct a $(3,1,6)$-graph $G$ of order $30$,
define the vertex set $V(G) = X \cup Y \cup Z$, where

\begin{equation*}
\begin{aligned}
X = & \{ x_i \, |\, 0 \leq i < 10 \} \\
Y = & \{ y_i \, |\, 0 \leq i < 10 \} \\
Z = & \{ z_i \, |\, 0 \leq i < 10 \}
\end{aligned}
\end{equation*}

The set of arcs $A(G)$ is
$$
\{ x_ix_{i+1} | 0 \leq i < 10 \} \cup\{ y_iy_{i+1} | 0 \leq i < 10 \} \cup\{ z_iz_{i+1} | 0 \leq i < 10 \}
$$

The set of edges $E(G)$ is
$$
\{ x_iy_{i\pm2} | 0 \leq i < 10 \} \cup \{ z_iz_{i+5} | 0 \leq i < 10 \}.
$$
In all cases, addition of subscripts is done modulo $10$.
\end{proof}

The graph described in the above proof is depicted in Figure~\ref{g316fig}.

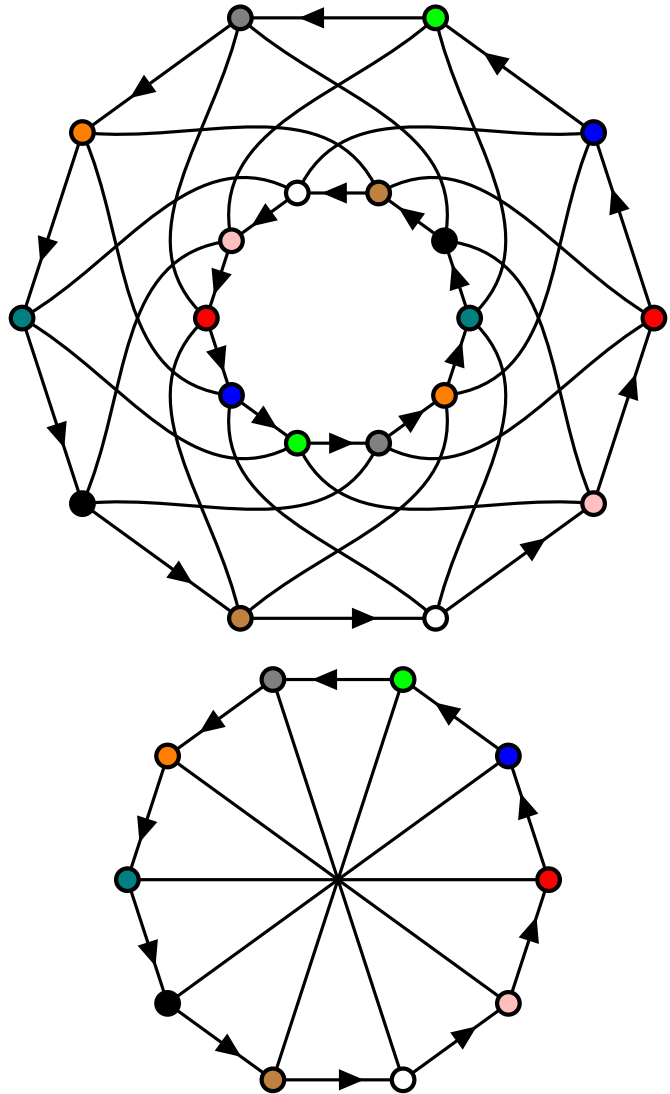
\begin{figure}[ht]
\centering
\setlength{\abovecaptionskip}{20pt}
\tikzset{fontscale/.style = {font=\relsize{#1}}}
\begin{tikzpicture}[scale=0.70,line width=1pt]

\tikzstyle{fred}=[draw=black, fill=yellow, ultra thick,
  shape=circle, minimum height=0.3cm, inner sep=1, text=black, font=\bfseries];
\tikzstyle{arc}=[very thick,
                 decoration={markings,mark=at position 0.70 with {\arrow[scale=1.0,>=triangle 45]{>}}},
                 postaction={decorate}];

\newcommand\orad{6}
\newcommand\mrad{2.5}

\foreach \rot in {0,36,72,...,324}
{
  \begin{scope}[rotate=\rot]
      \draw[arc] (0:\orad) -- (36:\orad);
      \draw[arc] (0:\mrad) -- (36:\mrad);
      \draw[line width=1.2,black] (0:\orad) to [out=150,in=30] (72:\mrad);
      \draw[line width=1.2,black] (0:\orad) to [out=210,in=330] (288:\mrad);
  \end{scope}
}

\foreach \rot/\c in {0/red,36/blue,72/green,108/gray,144/orange,180/teal,216/black,
                         252/brown,288/white,324/pink}
{
    \begin{scope}[rotate=\rot]
        \node[fred,fill=\c] at (0:\orad) [fontscale=2] {};
        \node[fred,fill=\c] at (180:\mrad) [fontscale=2] {};
    \end{scope}
}
\end{tikzpicture}


\hspace{12mm}

\begin{tikzpicture}[scale=0.70,line width=1pt]

\tikzstyle{fred}=[draw=black, fill=yellow, ultra thick,
  shape=circle, minimum height=0.3cm, inner sep=1, text=black, font=\bfseries];
\tikzstyle{arc}=[very thick,
                 decoration={markings,mark=at position 0.70 with {\arrow[scale=1.0,>=triangle 45]{>}}},
                 postaction={decorate}];

\newcommand\irad{4.0}

    \foreach \rot in {0,36,72,...,324}
    {
    \begin{scope}[rotate=\rot]
        \draw[arc] (0:\irad) -- (36:\irad);
    \end{scope}
    }
    \foreach \rot in {0,36,...,144}
    {
    \begin{scope}[rotate=\rot]
        \draw[line width=1.2,black] (0:\irad) -- (180:\irad);
    \end{scope}
    }
    \foreach \rot/\c in {0/red,36/blue,72/green,108/gray,144/orange,180/teal,216/black,
                         252/brown,288/white,324/pink}
    {
    \begin{scope}[rotate=\rot]
        \node[fred,fill=\c] at (0:\irad) [fontscale=2] {};
    \end{scope}
    }
\end{tikzpicture}
\caption{The unique smallest $(3,1,6)$-graph of order $30$.  Vertices in the lower figure are
adjacent to vertices in the upper figure that have the same color.}
\label{g316fig}
\end{figure}

Next we present a few constructions where we were able to come relatively close
to the AHM bound.

\begin{theorem} 
\label{thmsummary}
\begin{eqnarray*}
\text{a)} & 52 \leq f(3,1,7) & \leq 60  \\
\text{b)} & 74 \leq f(3,1,8) & \leq 76  \\
\text{c)} & 29 \leq f(4,1,5) & \leq 34  \\
\text{d)} & 46 \leq f(4,1,6) & \leq 48  \\
\text{e)} & 40 \leq f(5,1,5) & \leq 50  \\
\text{f)} & 66 \leq f(5,1,6) & \leq 72  
\end{eqnarray*}
\end{theorem}

All lower bounds are derived from the AHM bound.  The
upper bound constructions are handled separately.

{\bf (a) $ f(3,1,7) \leq 60$}

The directed subgraph for the $(3,1,7)$-graph consists of $6$ disjoint directed
$10$-cycles.  Each of these cycles is represented by a node in Figure~\ref{fig317}.
It will be useful to imagine the $60$ vertices of the graph labeled
$v_{i,j}$ for $0 \leq i < 6$ and $0 \leq j < 10$.  In this way, node $i$ in the
figure represents all vertices with first subscript $i$.  Each these sets of $10$ vertices
induces a directed $10$-cycle labeled in the natural order.

Various pairs of $10$-cycles are connected by undirected matchings.  The arcs in the
figure do not represent arcs in the graph, but are there to specify
the matching between pairs of $10$-cycles.
For example, the arc from node $1$ to node $2$ is labeled with
a $4$.  This means there is a matching connecting $10$-cycle number $1$ to $10$-cycle
number $2$ such that vertex $i$ in the first $10$-cycle is adjacent to vertex
$i+4$ in the second $10$-cycle, i.e., $v_{1,i}$ is adjacent to $v_{2,i+4}$, and again
subscript addition is modulo $10$.

The reader familiar
with voltage graphs can think of the arc labels as voltage assignments from the cyclic
group of order $10$.

\begin{figure}[ht]
\centering
\begin{tikzpicture}[scale=1.0,line width=1pt]
\tikzstyle{fred}=[draw=black, fill=yellow, 
  shape=circle, minimum height=3mm, inner sep=1, text=black, font=\bfseries];
\tikzstyle{arc}=[thick,
                 decoration={markings,mark=at position 0.85 with {\arrow[scale=0.6,>=triangle 45]{>}}},
                 postaction={decorate}];
\tikzstyle{edge}=[draw=black, thick];

\coordinate (v0) at (-1,0);
\coordinate (v1) at (-2,-1.732);
\coordinate (v2) at (-2, 1.732);
\coordinate (v3) at (1,0);
\coordinate (v4) at (2,-1.732);
\coordinate (v5) at (2, 1.732);

\draw[arc] (v0) -- (v1) node [midway,right] {0};
\draw[arc] (v0) -- (v2) node [midway,right] {0};
\draw[arc] (v0) -- (v3) node [midway,below] {0};
\draw[arc] (v1) -- (v2) node [midway,left] {4};
\draw[arc] (v1) -- (v4) node [midway,above] {6};
\draw[arc] (v2) -- (v5) node [midway,below] {4};
\draw[arc] (v3) -- (v4) node [midway,left] {0};
\draw[arc] (v3) -- (v5) node [midway,left] {0};
\draw[arc] (v4) -- (v5) node [midway,right] {6};

\node[fred] at (v0) {0};
\node[fred] at (v1) {1};
\node[fred] at (v2) {2};
\node[fred] at (v3) {3};
\node[fred] at (v4) {4};
\node[fred] at (v5) {5};

\end{tikzpicture}
\caption{Structure of an $(3,1,7)$-graph of order $60$.  Each node represents a directed $10$-cycle.}
\label{fig317}
\end{figure}
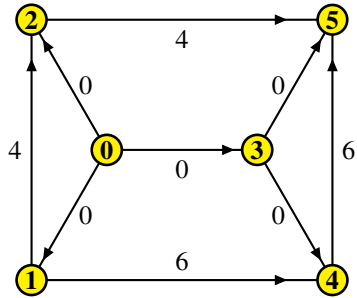

{\bf (b) $ f(3,1,8) \leq 76$}

Construct a $(3,1,8)$-graph $G$ on vertices $v_{i,j}$ for $0 \leq i < 2$ and
$0 \leq j < 38$.  For fixed $i$, the $38$ vertices $v_{i,j}$ form
a directed cycle, labeled in the natural order.
The edge set of $G$ is comprised of edges of one of the following forms:

\vspace{-18pt}
\begin{align*}
  v_{0,j}&v_{0,j+7}  \\[-8pt]
  v_{0,j}&v_{0,j-7}  \\[-8pt]
  v_{1,j}&v_{1,j+11} \\[-8pt]
  v_{1,j}&v_{1,j-11} \\[-8pt]
  v_{0,j}&v_{1,j}
\end{align*}

The automorphism group of $G$ has order $38$, with each directed $38$-cycle
comprising an orbit.

{\bf (c) $ f(4,1,5) \leq 34$}

An example $(4,1,5)$-graph of order $34$ is shown in Figure~\ref{fig415}.
The directed subgraph consists of two $17$-cycles, as shown in the right
half of the figure.  The undirected subgraph is shown in the left half
of the figure.  We identify the vertices in the left figure with the
corresponding vertices in the right figure.
Label the vertices
$v_{i,j}$ for $0 \leq i < 2$ and $0 \leq j < 17$, where we imagine
the vertices with first subscript $0$ are those on the outer circle.
Then there are arcs from $v_{0,i}$ to $v_{0,i+7}$ and from
$v_{1,i}$ to $v_{1,i+6}$, and edges $v_{0,j} v_{1,j-2}$ and $v_{0,j} v_{1,j+2}$.

\begin{figure}[ht]
\centering
\begin{tikzpicture}[scale=0.60,line width=1pt]
\tikzstyle{fred}=[draw=black, fill=yellow, thick,
  shape=circle, minimum height=2mm, inner sep=1, text=black, font=\bfseries];
\tikzstyle{arc}=[thick,
                 decoration={markings,mark=at position 0.95 with {\arrow[scale=0.6,>=triangle 45]{>}}},
                 postaction={decorate}];
\tikzstyle{edge}=[draw=black, thick];

\newcommand\rot{360.0/17.0};
\newcommand\sot{540.0/34.0};

\foreach \theta in {0,1,...,16}
{
    \begin{scope}[rotate=\theta*\rot]
            \draw[edge] (0:5) -- (\rot:5);
            \draw[edge] (0:2) -- (8*\rot:2);
            \draw[edge] (0:5) -- (2*\rot:2);
            \draw[edge] (0:5) -- (-2*\rot:2);
    \end{scope}
}

\foreach \theta in {0,1,...,16}
{
    \begin{scope}[rotate=\theta*\rot]
        \node[fred] at (0:5) {};
        \node[fred] at (0:2) {};
    \end{scope}
}

\begin{scope}[xshift=11cm]
\foreach \theta in {0,1,...,16}
{
    \begin{scope}[rotate=\theta*\rot]
            \draw[arc] (0:5) to [out=135,in=45-\sot] (7*\rot:5);
            \draw[arc] (0:2) -- (6*\rot:2);
    \end{scope}
    \foreach \theta in {0,1,...,16}
    {
        \begin{scope}[rotate=\theta*\rot]
            \node[fred] at (0:5) {};
            \node[fred] at (0:2) {};
        \end{scope}
    }
}
\end{scope}

\end{tikzpicture}
\caption{A $(4,1,5)$-graph of order $34$.  The figure on the left shows the edges
and the figure on the right shows the arcs.}
\label{fig415}
\end{figure}
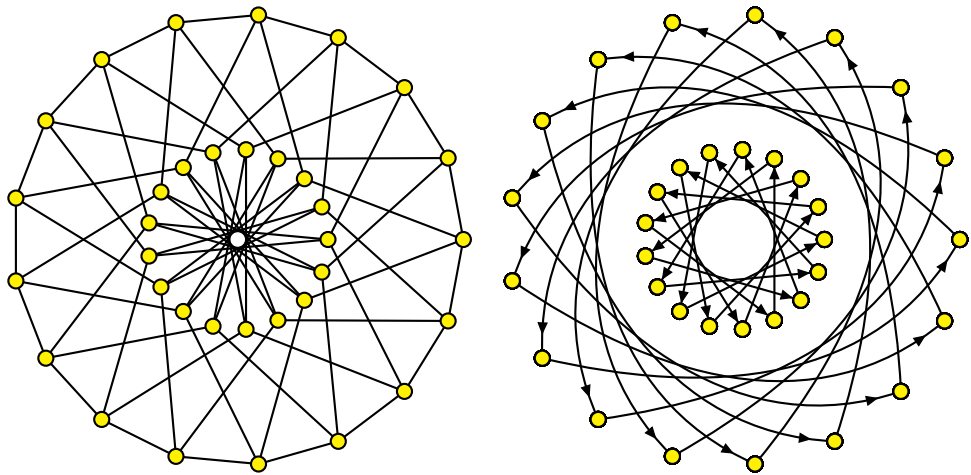

{\bf (d) $ f(4,1,6) \leq 48$}

For this case, refer to Figure~\ref{fig416}, which is similar to Figure~\ref{fig317} used
for part (a).
Here each node in the figure represents a directed $8$-cycle, and various of these
$8$-cycles are joined by undirected matchings, as in part (a).
Once again, the reader familiar
with voltage graphs can think of the labels as voltage assignments from the cyclic
group of order $8$.

\begin{figure}[ht]
\centering
\begin{tikzpicture}[scale=1.00,line width=1pt]
\tikzstyle{fred}=[draw=black, fill=yellow, 
  shape=circle, minimum height=3mm, inner sep=1, text=black, font=\bfseries];
\tikzstyle{arc}=[thick,
                 decoration={markings,mark=at position 0.85 with {\arrow[scale=0.6,>=triangle 45]{>}}},
                 postaction={decorate}];
\tikzstyle{edge}=[draw=black, thick];

\coordinate (v0) at (90:4);
\coordinate (v1) at (150:1);
\coordinate (v2) at (30:1);
\coordinate (v3) at (210:4);
\coordinate (v4) at (270:1);
\coordinate (v5) at (330:4);

\draw[arc] (v0) -- (v1) node [midway,left] {0};
\draw[arc] (v0) -- (v2) node [midway,right] {0};
\draw[arc] (v0) -- (v3) node [midway,left] {0};
\draw[arc] (v0) -- (v5) node [midway,right] {0};
\draw[arc] (v1) -- (v2) node [midway,above] {3};
\draw[arc] (v1) -- (v3) node [midway,left] {5};
\draw[arc] (v1) -- (v4) node [midway,left] {0};
\draw[arc] (v2) -- (v4) node [midway,right] {2};
\draw[arc] (v2) -- (v5) node [midway,right] {3};
\draw[arc] (v3) -- (v4) node [midway,below] {6};
\draw[arc] (v3) -- (v5) node [midway,below] {5};
\draw[arc] (v4) -- (v5) node [midway,below] {4};

\node[fred] at (v0) {0};
\node[fred] at (v1) {1};
\node[fred] at (v2) {2};
\node[fred] at (v3) {3};
\node[fred] at (v4) {4};
\node[fred] at (v5) {5};

\end{tikzpicture}
\caption{Structure of an $(4,1,6)$-graph of order $48$.  Each node represents a directed $8$-cycle.}
\label{fig416}
\end{figure}
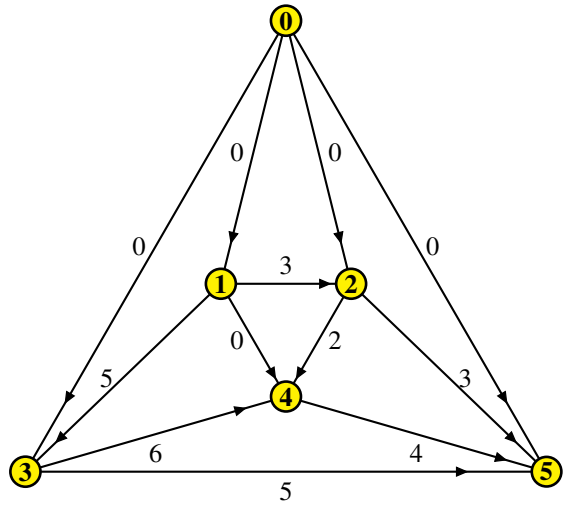

{\bf (e) $ f(5,1,5) \leq 50$}

A $(5,1,5)$-graph can be easily obtained from the Hoffman-Singleton graph by cyclically
orienting a directed Hamiltonian cycle, of which there are many.

{\bf (f) $ f(5,1,6) \leq 72$}

The graph is depicted in Figure~\ref{fig516}.  This time we replace vertices in the
base graph by $12$-cycles.

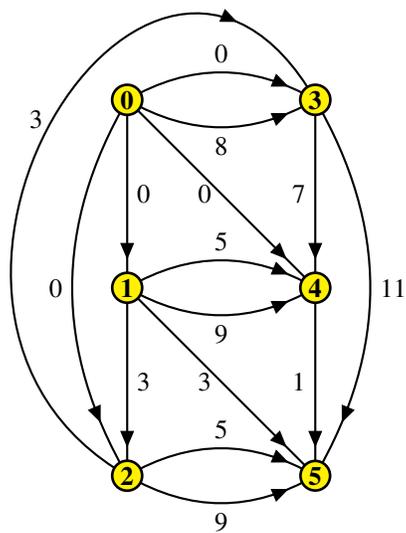
\begin{figure}[ht]
\hspace{38mm}
\begin{tikzpicture}[scale=1.25,line width=1pt]
\tikzstyle{fred}=[draw=black, fill=yellow, 
  shape=circle, minimum height=3mm, inner sep=1, text=black, font=\bfseries];
\tikzstyle{arc}=[thick,
                 decoration={markings,mark=at position 0.85 with {\arrow[scale=0.8,>=triangle 45]{>}}},
                 postaction={decorate}];
\tikzstyle{edge}=[draw=black, thick];

\coordinate (v0) at (0,4);
\coordinate (v1) at (0,2);
\coordinate (v2) at (0,0);
\coordinate (v3) at (2,4);
\coordinate (v4) at (2,2);
\coordinate (v5) at (2,0);

\draw[arc] (v0) -- (v1) node [midway,right] {0};
\draw[arc] (v0) to [out=240,in=120] node [midway,left] {0} (v2);
\draw[arc] (v1) -- (v2) node [midway,right] {3};
\draw[arc] (v3) -- (v4) node [midway,left] {7};
\draw[arc] (v3) to [out=300,in=60] node [midway,right] {11} (v5);
\draw[arc] (v4) -- (v5) node [midway,left] {1};

\draw[arc] (v0) -- (v4) node [midway,left] {0};
\draw[arc] (v1) -- (v5) node [midway,left] {3};

\draw[arc] (v0) to [out=30,in=150] node [midway,above] {0} (v3);
\draw[arc] (v0) to [out=330,in=210] node [midway,below] {8} (v3);
\draw[arc] (v1) to [out=30,in=150] node [midway,above] {5} (v4);
\draw[arc] (v1) to [out=330,in=210] node [midway,below] {9} (v4);
\draw[arc] (v2) to [out=30,in=150] node [midway,above] {5} (v5);
\draw[arc] (v2) to [out=330,in=210] node [midway,below] {9} (v5);

\draw[arc,looseness=2] (v2) to [out=150,in=120] node [midway,left] {3} (v3);

\node[fred] at (v0) {0};
\node[fred] at (v1) {1};
\node[fred] at (v2) {2};
\node[fred] at (v3) {3};
\node[fred] at (v4) {4};
\node[fred] at (v5) {5};

\end{tikzpicture}
\caption{Structure of a $(5,1,6)$-graph of order $72$.  Each node represents a directed $12$-cycle.}
\label{fig516}
\end{figure}

\clearpage
\newpage

\section{Further Work}
\label{sec:further}

There are at least four areas for further work on this topic.

\begin{enumerate}
\item
Table~\ref{tabr1g} lists cases where we were able to either determine an exact value for $f(r,z,g)$
or come reasonably close.  Note that for the case $(r,z,g) = (3,1,8)$ there are only two remaining
possibilities, either there is a graph achieving the AHM bound of $74$, or else the graph presented
here is a cage.  Perhaps techniques from Linear Algebra could be used to eliminate $74$ and settle
the issue.

\item
Recall that the directed subgraph of an $(r,1,g)$-graph consists of a set of directed cycles
whose lengths partition the order of the graph.
For most of the graph constructions presented here, these partitions contained equal parts.
At first, our computer searches were not restricted to partition with equal parts, but it
turned out that best constructions we found for smaller cases used such partitions.
So for cases where the search space was deemed too large for an exhaustive search,
we restricted the search to partitions with equal parts.  In each
of the cases listed in the table, there is no smaller graph with an equal partition.
Given enough computer time, one could probably complete exhaustive searches through all
partitions and settle the unresolved cases listed here.

\item
In the section on $r=2$, three theorems were given providing upper bounds for $f(2,z,g)$.
Now the undirected subgraph of a $(2,z,g)$-graph consists of disjoint cycles, i.e.,
(undirected) cages of degree $2$.  The reader will notice that these theorems
could each be generalized for $r > 2$, replacing the cycles by $r$-cages.

\item
As noted above, several of the constructions in this paper can be viewed as voltage
graphs over cyclic groups.  But they were not found by a search for voltage graphs,
so they were not presented as such.  Voltage graphs over more interesting groups
might be a good place to look for $(r,z,g)$-cages.
\end{enumerate}

\begin{table}[h]
\centering
\renewcommand{\arraystretch}{1.2}
\begin{tabular}{|r|r|r||r|r|r|} \hline
 r & z & g & Lower & Exact & Upper \\ \hline \hline
 2 & 2 & 5 &    & 19 &    \\ \hline
 2 & 2 & 6 &    & 27 &    \\ \hline
 3 & 1 & 5 &    & 24 &    \\ \hline
 3 & 1 & 6 &    & 30 &    \\ \hline
 3 & 1 & 7 & 52 &    & 60 \\ \hline
 3 & 1 & 8 & 74 &    & 76 \\ \hline
 4 & 1 & 5 & 29 &    & 34 \\ \hline
 4 & 1 & 6 & 46 &    & 48 \\ \hline
 5 & 1 & 5 & 40 &    & 50 \\ \hline
 5 & 1 & 6 & 66 &    & 72 \\ \hline
\hline
\end{tabular}
\caption{Bounds for small $f(r,z,g)$.
For each of the values of $r$, $z$, and $g$,
either the exact value, or our best lower and upper bounds are given.
The result for $(2,2,5)$ is due to
Claudia De La Cruz and Miguel Piza\~{n}a (unpublished).
}
\label{tabr1g}
\end{table}

\clearpage
\newpage

\nocite{*}
\bibliographystyle{abbrvnat}
\bibliography{mixedfinal}
\label{sec:biblio}

\end{document}